\documentclass[preprint,12pt,3p]{elsarticle}
\usepackage{verbatim}
\usepackage{amsmath,amsthm,amsfonts,amssymb}
\numberwithin{equation}{section}
\usepackage{natbib}
\usepackage{mathrsfs}
\def\BState{\State\hskip-\ALG@thistlm}
\makeatother

\usepackage{mathtools}
\usepackage{comment}
\usepackage{cases}

\usepackage{hyperref}
\usepackage[normalem]{ulem}
\usepackage{xcolor,sectsty}
\usepackage{subfigure}
\newtheorem{theorem}{Theorem}[section]
\newtheorem{lemma}[theorem]{Lemma}
\newtheorem{definition}[theorem]{Definition}
\newtheorem{example}[theorem]{Example}

\newtheorem{corollary}[theorem]{Corollary}
\newtheorem{proposition}[theorem]{Proposition}
\journal{Arxiv}

\newcommand{\mc}{{\mathcal{R}}}
\newcommand{\mci}{{\mathcal{R}^{-1}}}

\newcommand{\core}{\scriptsize\mbox{\textcircled{\#}}}

\begin{document}

\begin{frontmatter}
\title{
{\bf Further results on weighted core inverse in a ring
}}

\author{ Sourav Das$^a$, Jajati Keshari Sahoo$^b$,   Ratikanta Behera$^c$}
\vspace{.3cm}

\address{

               $^a$Department of Mathematics,\\ National Institute of Technology Jamshedpur, Jharkhand-831014, India.\\
                        \textit{E-mail$^a$}: \texttt{souravdasmath@gmail.com, souravdas.math@nitjsr.ac.in}

               \vspace{.3cm}
                        $^{b}$Department of Mathematics,\\
                       Birla Institute of Technology $\&$ Science, Pilani K. K. Birla Goa Campus, India
                        \\\textit{E-mail}: \texttt{jksahoo\symbol{'100}goa.bits-pilani.ac.in}

               \vspace{.3cm}
               $^{c}$Department of Mathematics,\\
                  University of Central Florida, Orlando, USA.\\
                        \textit{E-mail}: \texttt{ratikanta.behera@ucf.edu}

                        }

\begin{abstract}
The notion of the weighted core inverse in a ring with involution was introduced, recently [Mosi\'c {\em et al.} Comm. Algebra, 2018; 46(6); 2332-2345]. In this paper, we explore new representation and characterization of the weighted core inverse of sum and difference of two weighted core invertible elements in  ring with involution under different conditions. Further, we discuss reverse order laws and mixed-type reverse order laws for the weighted core invertible elements in a ring.
\end{abstract}


\end{frontmatter}
\section{Introduction}
Core inverses, introduced in \cite{baks,BakTr14}, have been frequently investigated over the past few years \cite{zhou2019core,Zhou-CA-2020}. Indeed, right weak generalized inverse proposed earlier in \cite{BenIsrael03,Cline68}, which was renamed as the core inverse. The authors of \cite{rakicAMC} have extended the concept of core inverse to the Hilbert space operators. Raki\'c {\em et al.} in \cite{Rakic2014} discussed the notion core and dual core inverse in rings with involution, then, a few characterizations of core inverses discussed in \cite{Zhu2019}. In generalizations of the core inverses and the dual core inverses; Mosi\'{c} {\em et al.} in \cite{Mosic-CA-2018} introduced the definitions of the weighted core inverse and the weighted dual core inverse in a ring with involution. More results on the weighted core inverse can be found in \cite{Li2018}. The vast literature on the core and weighted core inverses \cite{Chen2017} with its multifarious extensions in different areas of mathematics \cite{Djordjevic2005}, motivate us study more new representation and characterization of the weighted core inverse in a ring.\\
Main contributions of this manuscript are as follows.
\begin{enumerate}
\item[$\bullet$] Derived a few explicit expressions for the 
sum and difference of two weighted core (weighted dual core) invertible elements. 
\item[$\bullet$] A few necessary and sufficient  conditions for reverse order are established.
\end{enumerate}
The explicit expression of sum and difference elements of generalized inverses plays a significant role in applied and computational mathematics \cite{Gonz_lez_2004,Djordjevi__2002}. Therefore, many researchers pay attention to investigate additive properties in different fields, including operators \cite{CZW, Cvetkovi_Ili__2006}, matrices \cite{Gonz_lez_2004}, tensors \cite{SahBe20} and elements of rings with involution \cite{Zhou-CA-2020}. In the paper, we derive explicit expression for the weighted core inverse of the sum two weighted core invertible elements in a ring  under various conditions.

It is natural to ask in the theory of generalized inverse  that, when the ordinary reverse order law 
\begin{equation*}
    (AB)^{-1} =B^{-1}A^{-1} 
\end{equation*}
can be applied in the generalized inverse case. The answer to this query was first studied in \cite{Grev1966} with a few necessary and sufficient conditions on the reverse order law for the Moore-Penrose inverse in from of rectangular matrices \cite{SunWei02, YiminWei98}. From that time, reverse order law for generalized inverses have significantly impacted many areas of science and engineering in the context of matrices \cite{Cao2004}, operators \cite{Deng11}, tensors \cite{PaniBeheMi20,JR_Rev20} and elements of rings with involution \cite{Zou-MedJM-2018,MosiDij12}. In particular, Koliha et al. \cite{Koliha2007} discussed the reverse order law for the Moore–Penrose invertible elements and Liu et al. \cite{Xia2014} derived some equivalences of the reverse order law for the group invertible elements in a ring. Further, the authors of \cite{Zou-MedJM-2018} have discussed a few one-sided reverse order law and two-sided reverse order law for the core inverse in rings. The vast work on the reverse order law \cite{MR2831656,MR2805535,SunWei02} focuses our attention to discuss reverse order laws for weighted core and dual inverse in a ring.

This paper is organized as follows. Necessary definitions and preliminaries results along with some results for weighted core inverse in a ring are discussed in Section 2. A few additive properties of the weighted core and dual core inverses over a ring have been presented in Section 3. In Section 4, we discuss the reverse order law for weighted core and dual core inverses.

\section{Preliminaries}
Let $\mathcal{R}$ be an unital ring with involution, $x\rightarrow x^\ast$ satisfying the following
\begin{equation*}
(xy)^{\ast}=y^{\ast}x^{\ast},~~(x+y)^{\ast}=x^{\ast}+y^{\ast}\mbox{~~ and~~}(x^{\ast})^{\ast}=x\mbox{~~ for all~~ }x, y\in \mathcal{R}.
\end{equation*}
Let us recall the definition of the Dedekind-finite ring, i.e., an unital ring $\mathcal{R}$ is said to be Dedekind-finite ring if $xy=1$ implies $xy=1$ for all $x,y\in\mathcal{R}.$
Now we recall the core inverse which was introduced in \cite{Rakic2014}.  Let $a\in\mathcal{R}$. If  an element $z\in\mc$ satisfies
\begin{equation*}
(1)~~aza=a,~~(2)~~zaz=z,~~(3)~~ (az)^*=az,~~(6)~~ za^2=a,~~(7)~~az^2=z,
\end{equation*}
then $z$ is called the core inverse of $a$. 
However, Xu {\em et al.} in \cite{Rakic2014} discussed the efficient representation of core inverses, which reduces to $$(3)~~ (az)^*=az,~~(6)~~ za^2=a,~~(7)~~az^2=z.$$
The Drazin and weighted core inverse of an element $a \in \mathcal{R}$ defined \cite{Drazin58} as follows.
\begin{definition}\cite{Drazin58}
Let $a \in \mathcal{R}$. An element  $z \in \mathcal{R}$ satisfying
\begin{equation*}
\left(6^k\right)~za^{k+1}=a^k,~~ (2)~zaz =z,~~(5)~az= za, \mbox{ for some } k\geq1,
\end{equation*}
is known as the Drazin inverse of $a$ and is denoted by $a^D$. 
\end{definition}
The group inverse $a^{\#}$ of $a$ is a special case of $a^D$ when $k=1$. 
The $m$-weighted core inverse and $n$-weighted dual core inverse are defined in \cite{Mosic-CA-2018}, as follows. 
\begin{definition}\label{mcore}\cite{Mosic-CA-2018}
Let $a\in\mathcal{R}$ and $m\in \mathcal{R}$ be an invertible element with $m^*=m$. Then an element $z\in \mathcal{R}$
is said to be $m$-weighted core inverse if
$$
aza = a,~~z\mathcal{R} = a\mathcal{R},~~\mbox{ and }~~\mathcal{R}z = \mathcal{R}a^*m. 
$$
The $m$-weighted core inverse of an element $a\in \mathcal{R}$ is denoted by $a^{\core,m}$ and unique (if exists).
\end{definition}

\begin{definition}\label{ncore}\cite{Mosic-CA-2018}
Let $a\in\mc$ and $n\in \mathcal{R}$ be an invertible element with $n^*=n$. If an element $y\in \mathcal{R}$ satisfies 
$$
aya = a,~~ny\mathcal{R} = a^*\mathcal{R},~~\mbox{ and }~~\mathcal{R}y = \mathcal{R}a, 
$$
then $y$ is called $n$-weighted dual core inverse of $a$.

\end{definition}
The $n$-weighted dual core inverse is unique (if exists) and it is denoted by $a_{n,\core}$. In particular, when $m=1$, the $a^{\core,m}$ reduces to the core inverse, i.e.,  $a^{\core}$. Similarly, for $n=1$, $a_{n,\core}$ is known as the dual core inverse of $a$ and is denoted by $a_{\core}$.


\subsection{Representation of weighted core inverse}
For convenience, we use $\mathcal{R}^{-1}$, $\mathcal{R}^{\#},\mathcal{R}^{\core}$, $\mathcal{R}^{\core, m}$ and $\mathcal{R}_{n,\core}$ denote the set of all invertible, group, core, $m$-weighted core and $n$-weighted dual core invertible elements of $\mathcal{R}$, respectively.  In addition to these, we use the notation $a^{(\lambda)}$ for an element of $\{\lambda \}$-inverses of $a$ and $a\{\lambda\}$ for the class of
$\{\lambda \}$-inverses of $a$, where $\lambda\in \{1, 2, 3, \cdots \}$. For instance, an element $x \in R$ is called a {\it $1$-inverse} of $a \in R$ if $x$ satisfies (1), i.e., $axa=a$ and we denote $x$ by $a^{(1)}$.
The following results will be helpful to prove some of our the main results.
\begin{lemma}\cite{Hartwig1977}\label{lm:group-inv}
Let $a\in\mathcal{R}.$ Then $a$ is group invertible {\em iff}
\begin{equation*}
a=a^2s=ta^2 \textnormal{~~for some~~} s,t\in \mathcal{R}.    
\end{equation*}
Moreover, $a^{\#}=yas=t^2a=as^2.$
\end{lemma}

\begin{theorem}\cite{Mosic-CA-2018}\label{lm:a-n-core}
Let $m\in \mathcal{R}^{-1}$ be a hermitian element. If $a\in\mathcal{R}^{\#}\cap\mathcal{R}^{(1,3^m)}=\phi$, then
$$
(a^p)^{\core,m}=(a^{\core,m})^{p} \mbox{ for any }p\in\mathbb{N}.
$$
\end{theorem}

\begin{theorem}\cite{Mosic-CA-2018}\label{core-equivalent}
Let $a\in \mathcal{R}$ and $m\in \mathcal{R}^{-1}$ with $m^*=m$. Then the following  hold.
\begin{enumerate}[(i)]
\item $a\in \mathcal{R}^{\core,m}$  and $x=a^{\core,m}.$
\item $ xax=x,~~~ax^2 = x,~~~(max)^* = max,~~~xa^2 = a~~axa=a.$
\item 
$
(max)^* = max,~~~xa^2 = a,~~~ax^2 = x.
$
\item $a^{\core,m}=a^{\#}aa^{(1,3^m)}$, for any $a^{(1,3^m)}\in a\{1,3^m\}.$
\item $a^{\#}=\left(a^{\core,m}\right)^2a.$
\item $a\in \mathcal{R}^{\#}\cap \mathcal{R}^{(1,3^m)}$.
\end{enumerate}
\end{theorem}

\begin{theorem}\cite{Mosic-CA-2018}\label{n-core-equivalent}
Let  $a\in \mathcal{R}$ and $n\in \mathcal{R}^{-1}$ with $n^*=n$. Then the following are equivalent.
\begin{enumerate}[(i)]
\item $a\in \mathcal{R}_{n,\core}$ and $y=a_{n,\core}$.
\item 
$
aya=a,~~~ yay=y,~~~(nya)^* = nya,~~~a^2y = a,~~~y^2a = y.
$
\item 
$
(nya)^* = nya,~~~a^2y = a,~~~y^2a = y.
$
\item $a\in \mathcal{R}^{\#}\cap \mathcal{R}^{(1,4^n)}$.
\item $a_{n,\core}=a^{(1,4^n)}aa^{\#}$, for any $a^{(1,4^n)}\in a\{1,4^n\}$.
\item $a^{\#}=a\left(a_{n,\core}\right)^2$.
\end{enumerate}
\end{theorem}
In view of Theorem \ref{core-equivalent}, we obtain the following equivalent representations for $m$-weighted core inverse.
\begin{lemma}\label{lm:m-weighted-core-alt}
Let $m\in \mathcal{R}^{-1}$ be a hermitian element and  $a\in\mathcal{R}$. If an element $x\in\mc$ satisfies 
$axa=a,~(max)^*=max,$ and $ax^2=x$, then $x=a^{\core,m}$.
\end{lemma}
\begin{proof}
Let  $ axa=a$ and $ax^2= x$. Then $ a= a^2x^2a,$ which yields  $ a^{\#}  a= a^{\#}a^2x^2a= ax^2a= xa.$ Now $xa^2= a^{\#}a^2=a.$ Hence ${x}$ is the $m$-weighted core inverse of the element $ a$.
\end{proof}

\begin{proposition}\label{cor2.10}
Let $m \in \mathcal{R}^{-1}$ be a hermitian element and $a\in \mathcal{R}^{\#}$. Further, let $x\in R$ with $(max)^*=max$. If $x$ satisfies either $xa^2=a$ or $xa=a^{\#}a$, then $a^{\core,m}=a^{\#}ax$.
\end{proposition}
\begin{proof}
Let  $xa^2=a$. Then     $axa=axaa^{\#}a=axa^2a^{\#}=a^2a^{\#}=a$.
Thus $a\in \mathcal{R}^{\#}\cap \mathcal{R}^{(1,3^m)}$. Hence by Theorem \ref{core-equivalent} (iv), $a^{\core,m}=a^{\#}ax$. Similarly, using $xa=a^{\#}a$, we can show $x\in a\{1,3^m\}$. Which again leads $a^{\core,m}=a^{\#}ax$.
\end{proof}

The above proposition is very useful to compute $m$-weighted core inverse of an element as shown in the next example.
\begin{example}\rm
Let $R=M_2(\mathbb{R})$, $a=\begin{pmatrix}
1  & 0\\
-1 & 0
\end{pmatrix}$ and $m=\begin{pmatrix}
2  & 1\\
1 & 2
\end{pmatrix}$. If we take  $x=\begin{pmatrix}
0.5 & -0.5\\
-0.5 & 0.5
\end{pmatrix}$, then we can verify that $max$ is symmetric and $xa=a^{\#}a$. Hence by Proposition \ref{cor2.10}, $x=a^{\core,m}.$ 
\end{example}

Using Proposition \ref{cor2.10}, we obtain another  representations for $m$-weighted core inverse, as follows.

\begin{lemma}\label{prop11}
Let $m\in \mathcal{R}^{-1}$ be a hermitian element and $a\in \mathcal{R}$. If an element $x\in\mc$ satisfies
$xax=x,~(max)^*=max$, and $xa^2=a$. then $x=a^{\core,m}$.
\end{lemma}
\begin{proof}
Let $a=xa^2$. Then $aa^{\#}= xa^2a^{\#}= xa$. Now $axa= a^2a^{\#}= a$. This yields $ x\in a\{1,3^m\}.$ Applying Proposition \ref{cor2.10}, we obtain $ a^{\core, m}= a^{\#} ax= aa^{\#}x$. Replacing $aa^{\#}$ by $ xa,$ we obtain $ a^{\core, m}= xax=x.$
\end{proof}
Using Lemma \ref{lm:m-weighted-core-alt} and \ref{prop11}, the following corollary can be easily proved.
\begin{corollary}\label{pro2.2}
For $a\in \mathcal{R}$, if $x\in a\{6,7\}$, then $x\in a\{1,2\}$.
\end{corollary}
The following results can be proved for $n$-weighted core dual inverse, 
\begin{lemma}\label{lm:n-weighted-core-alt}
Let $n\in \mathcal{R}^{-1}$ be a hermitian element and  $a\in\mathcal{R}$. If an element $y\in\mc$ satisfies $aya=a,~(nya)^*=nya,$ and $y^2a=y$, then $y=a_{n,\core}$. 
\end{lemma}

\begin{proposition}\label{ncor2.10}
Let $n \in \mathcal{R}^{-1}$ be a hermitian element and $a\in \mathcal{R}^{\#}$. Further, let $y\in R$ with $(nya)^*=nya$. If $y$ satisfies either $a^2y=a$ or $ay=aa^{\#}$, then $a_{n,\core}=yaa^{\#}$.
\end{proposition}
Likewise the $m$-weighted core inverse, the above proposition will be useful for computing $a_{n,\core}.$ 
\begin{example}\rm
Let $R=M_2(\mathbb{R})$, $a=\begin{pmatrix}
1  & 0\\
-1 & 0
\end{pmatrix}$ and $m=\begin{pmatrix}
1  & 0\\
0 & 2
\end{pmatrix}$. If we take  $y=\begin{pmatrix}
1 & 0\\
0 & 0
\end{pmatrix}$, then we can verify that $nya$ is symmetric and $ay=aa^{\#}$. Hence by Proposition \ref{ncor2.10}, $y=a_{n,\core}$.
\end{example}
\begin{lemma}\label{nprop11}
Let $n\in \mathcal{R}^{-1}$ be a hermitian element and $a\in \mathcal{R}$. If an element satisfies $yay=y,~(nya)^*=nya$, and $a^2y=a$, then  $y=a_{n,\core}$.
\end{lemma}
\begin{corollary}
For $a\in \mathcal{R}$, if $y\in a\{8,9\}$, then $y\in a\{1,2\}$.
\end{corollary}

In view of Lemma \ref{lm:m-weighted-core-alt} and Lemma \ref{prop11}, we obtain the following equivalent characterization for weighted core inverse. 
\begin{theorem}\label{thm:m-core-eqv}
Let $a,x\in\mathcal{R}$ and $m\in\mathcal{R}^{-1}$ be such that $m^*=m$. Then the following assertions are equivalent:
\begin{enumerate}[(i)]
\item\label{itm1} $a \in\mathcal{R}^{\core,m}$ and $z=a^{\core,m}$.
\item\label{itm2} $a\in \mathcal{R}^{\#},$ $aza=a,$ $(maz)^{*}=maz$, and $z\mathcal{R}\subseteq a\mathcal{R}.$
\item\label{itm3} $a\in \mathcal{R}^{\#},$ $zaz=z,$ $(maz)^{*}=maz$, and $a\mathcal{R}\subseteq z\mathcal{R}.$
\end{enumerate}
\end{theorem}

Similarly, the following theorem can be established for $n$-weighted dual core inverse. 
\begin{theorem}\label{thm:n-core-eqv}
Let $a,y\in\mathcal{R}$ and $n\in\mathcal{R}^{-1}$ be a hermitian element. Then the following are equivalent:
\begin{enumerate}[(i)]
\item $a \in\mathcal{R}_{n,\core}$ and $y=a_{n,\core}$.
\item $a\in \mathcal{R}^{\#},$ $aya=a,$ $(nya)^{*}=nya$, and $\mathcal{R} y\subseteq \mathcal{R} a.$
\item $a\in \mathcal{R}^{\#},$ $yay=y,$ $(nya)^{*}=nya$, and $\mathcal{R} a\subseteq \mathcal{R} y.$
\end{enumerate}
\end{theorem}
We now discus a characterization for $m$-weighted core inverse. 
\begin{theorem}
Let $m\in\mathcal{R}^{-1}$ be a hermitian element, $a\in\mathcal{R}^{\#}$ and $ab\in\mathcal{R}^{\core,m}$. Then the relation
$a\mathcal{R}\subseteq bab \mathcal{R}$ is valid if and only if $a\in\mathcal{R}^{\core,m}$ and $a^{\core,m}=b(ab)^{\core,m}.$
\end{theorem}
\begin{proof}
Using $a\mathcal{R}\subseteq bab \mathcal{R}$, we get $a=babt$ for some $t\in \mathcal{R}$. Now  $a=b(ab)^{\core,m}(ab)^2t$. Therefore, $a\mathcal{R}\subseteq b(ab)^{\core,m}\mathcal{R}.$
Further, it can be verified that $b(ab)^{\core,m}\in a\{2,3^m\}.$ Finally, using Theorem \ref{thm:m-core-eqv}, we obtain that
$a\in \mathcal{R}^{\core,m}$ with $a^{\core,m}=b(ab)^{\core,m}.$

Conversely, assume that $a\in\mathcal{R}^{\core,m}$ and $a^{\core,m}=b(ab)^{\core,m}.$ Then $a\mathcal{R}\subseteq bab \mathcal{R}$ is follows from  
\begin{equation*}
  a=a^{\core,m}a^2=b(ab)^{\core,m}a^2=bab\left((ab)^{\core,m}\right)^2. \qedhere 
\end{equation*}
\end{proof}

\section{Additive properties of weighted core inverse}\label{a}
In this section, we discuss about the existence of $(a+b)^{\core,m}$ and $(a-b)^{\core,m}$, where $a,b\in\mathcal{R}^{\core,m}$. In addition, the explicit expression for $(a+b)^{\core,m}$ and $(a-b)^{\core,m}$ are presented. 
We begin with a lemma  which will be helpful to prove the required results.
\begin{lemma}\cite{CZW} \label{additive-group}
Let $c,~d\in \mathcal{R}^{\#}$ and $cd=0$. Then $(c+d)\in\mathcal{R}^{\#}$. Moreover, $$(c+d)^\#=(1-dd^\#)c^\#+d^\#(1-cc^\#).$$
\end{lemma}

The first result of additive properties for $m$-weighted core inverse is presented below.

\begin{theorem}
\label{core-additive1}
Let $a,b\in \mathcal{R}^{\core,m}$ and $m\in\mathcal{R}^{-1}$ be a hermitian element.   If $a^{\ast}mb=0$ and $ab=0$, then
$a+b\in \mathcal{R}^{\core,m}$   and
$$
(a+b)^{\core,m}=(1-b^{\core,m}b)a^{\core,m}+b^{\core,m}.
$$
\end{theorem}
\begin{proof} 
Let $a,b\in \mathcal{R}^{\core,m}$. Then by Theorem \ref{core-equivalent}, we obtain $a,b\in \mathcal{R}^\#$. Further,
\begin{equation}\label{eqn3.1}
 b^\#=(b^{\core,m})^{2}b~\mbox{ and }~a^\#=(a^{\core,m})^{2}a.
\end{equation}
 Applying equation \eqref{eqn3.1} and Lemma \ref{additive-group} to the elements $a$ and $b$, we have 
 \begin{eqnarray}\label{eqq3.2}
\nonumber
(a+b)^\#
&=(1-(b^{\core,m})^{2}b^{2})(a^{\core,m})^{2}a+
(b^{\core,m})^{2}b(1-(a^{\core,m})^{2}a^{2}).\\
&=(1-b^{\core,m}b)(a^{\core,m})^{2}a+
(b^{\core,m})^{2}b(1-a^{\core,m}a).
\end{eqnarray}
From  $a^{\ast}m b=0$ and $ab=0$, we get $ab^{\core,m}=ab(b^{\core,m})^{2}=0.$ In addition, 
\begin{equation*}
\begin{split}
b^{\core,m}a
&=b^{\core,m}bb^{\core,m}a
=b^{\core,m}m^{-1}(mbb^{\core,m})^{\ast}a
=b^{\core,m}m^{-1}(b^{\core,m})^{\ast}b^{\ast}m a\\
&=b^{\core,m}m^{-1}(b^{\core,m})^{\ast}(a^{\ast}mb)^{\ast}
=0,
\end{split}
\end{equation*}
\begin{equation*}
 a^{\core,m}b
=a^{\core,m}aa^{\core,m}b
=a^{\core,m}m^{-1}(maa^{\core,m})^{\ast}b
=a^{\core,m}m^{-1}(a^{\core,m})^{\ast}a^{\ast}mb
=0.   
\end{equation*}
Let $y=(1-b^{\core,m}b)a^{\core,m}+ b^{\core,m}.$ Then
\begin{eqnarray}\label{eq3.1}
\nonumber
m(a+b)y
&=&m(a+b)[(1-b^{\core,m}b)a^{\core,m}+
b^{\core,m}]\\
\nonumber
&=&ma(1-b^{\core,m}b)a^{\core,m}+mb(1-b^{\core,m}b)a^{\core,m}+
mab^{\core,m}+mbb^{\core,m}\\
&=&ma(1-b^{\core,m}b)a^{\core,m}+mbb^{\core,m}
=maa^{\core,m}+mbb^{\core,m}.
\end{eqnarray}
Which implies
\begin{center}
  $\left(m(a+b)y \right)^*=\left(maa^{\core,m} \right)^*+\left(mbb^{\core,m}\right)^{*}=maa^{\core,m}+mbb^{\core,m}
=m(a+b)y.$  
\end{center}
Using equation \ref{eq3.1} (for $m=1$), we have
\begin{equation*}
\begin{split}
(a+b)y(a+b)
&=(aa^{\core,m}+bb^{\core,m})(a+b)
= aa^{\core,m}a+aa^{\core,m}b+bb^{\core,m}a+bb^{\core,m}b\\
&=aa^{\core,m}a+bb^{\core,m}b=a+b.
\end{split}
\end{equation*}
Therefore, $y$ is an $\{1,3^m\}$-inverse of the element $a$. Using Theorem \ref{core-equivalent} (iv) and equation \eqref{eqq3.2}, we have
\begin{equation*}
\begin{split}
(a+b)^{\core,m}&=\left[(1-b^{\core,m}b)(a^{\core,m})^{2}a+(b^{\core,m})^{2}b(1-a^{\core,m}a)
\right]\times\\
&\qquad (a+b)\left[(1-b^{\core,m}b)a^{\core,m}+b^{\core,m}
\right]\\
&=\left[\right.(1-b^{\core,m}b)(a^{\core,m})^{2}a^{2}
+(1-b^{\core,m}b)(a^{\core,m})^{2}ab+
(b^{\core,m})^{2}b(1-a^{\core,m}a)a\\
&\hspace*{0.5cm}+
(b^{\core,m})^{2}b(1-a^{\core,m}a)b\left.\right]
\left[(1-b^{\core,m}b)a^{\core,m}+
b^{\core,m}\right]\\
&=\left[(1-b^{\core,m}b)a^{\core,m}a
+b^{\core,m}b\right]\left[(1-b^{\core,m}b)a^{\core,m}+
b^{\core,m}\right]\\
&=(1-b^{\core,m}b)a^{\core,m}a(1-b^{\core,m}b)a^{\core,m}
+b^{\core,m}b(1-b^{\core,m}b)a^{\core,m}\\
&\qquad+(1-b^{\core,m}b)a^{\core,m}ab^{\core,m}+
b^{\core,m}bb^{\core,m}\\
&=(1-b^{\core,m}b)a^{\core,m}+b^{\core,m}.
\qedhere
\end{split}
\end{equation*}
\end{proof}
If we drop the  condition $ab=0$, then the theorem may not be valid in general. One such example is given below. 
\begin{example}\rm
Let $R=M_2(\mathbb{R})$, $a=\begin{pmatrix}
1  & 2\\
0 & 0
\end{pmatrix}$, $b=\begin{pmatrix}
-1  & 3\\
0 & 0
\end{pmatrix}$, and $m=\begin{pmatrix}
1  & 0\\
0 & 5
\end{pmatrix}$. Since $a=a^2$ and $b=-b^2$, it follows that $a$ and $b$ both are group invertible. Clearly the element  $x=\begin{pmatrix}
1  & 0\\
0 & 0
\end{pmatrix}$ satisfies $(max)^*=max$ and $xa^2=a$. Hence by Proposition \ref{cor2.10}, $a\in\mc^{\core,m}$ and $a^{\core,m}=a^{\#}ax$. Using Proposition \ref{ncor2.10}, we can show $b\in\mc^{\core,m}$. As $a+b=\begin{pmatrix}
0  & 5\\
0 & 0
\end{pmatrix}$  and $(a+b)^2=\begin{pmatrix}
0  & 0\\
0 & 0
\end{pmatrix}$, so $(a+b)\notin \mc^{\#}$. Hence $(a+b)\notin\mc^{\core,m}$.
\end{example}

If we impose an additional condition $ba=0$, then we obtain the following result as a corollary.

\begin{corollary}
\label{core-additive2}
Let $a,b\in \mathcal{R}^{\core,m}$ and $m\in\mathcal{R}^{-1}$ be a hermitian element.  If $a^{\ast}mb=0=ab$ and $ba=0$ then
$a+b\in \mathcal{R}^{\core,m}$ and
$$
(a+b)^{\core,m}=a^{\core,m}+b^{\core,m}.
$$
\end{corollary}
In case of $n$-weighted dual core inverse, the following results can be established. 
\begin{theorem}
Let $a,b\in\mathcal{R}$ and $n\in\mci$ such that $an^{-1}b^{\ast}=0=ab,$ $n^*=n$. If $a,b\in \mathcal{R}_{n,\core}$, then
$a+b\in \mathcal{R}_{n,\core}$   and
$$
(a+b)_{n,\core}=a_{n,\core}+b_{n,\core}(1-aa_{n,\core}).
$$
\end{theorem}
\begin{corollary}
Let $a,b\in\mathcal{R}$  and $n\in\mci$ such that $ab=0$, $an^{-1}b^{\ast}=0=ab$, $n^*=n$. If $a,b\in \mathcal{R}_{n,\core}$, then
$a+b\in \mathcal{R}_{n,\core}$   and
$$
(a+b)_{n,\core}=a_{n,\core}+b_{n,\core}.
$$
\end{corollary}

With help of the conditions $a=ba^{\core,m}a=a$, and $2\in \mathcal{R}^{-1}$, an additional representation for $m$-weighted core inverse  of sum of two elements discussed below. 

\begin{theorem}\label{thm:sum-alternative}
Let $2\in \mathcal{R}^{-1}$, $a,b\in\mathcal{R}^{\core,m}$ and $m\in \mathcal{R}^{-1}$ be a hermitian element.
If $ba^{\core,m}a=a,$ then $a+b\in \mathcal{R}^{\core,m}$  and
$$
(a+b)^{\core,m}=-\dfrac{1}{2}a^{\core,m}+b^{\core,m}+\dfrac{1}{2}a^{\core,m}bb^{\core,m}-\dfrac{1}{2}a^{\core,m}ab^{\core,m}.
$$
\end{theorem}
\begin{proof}
Let $bb^{\core,m}a=a$. Then 
\begin{align}\label{424}
ab^{\core,m}=ba^{\core,m}ab^{\core,m},~~ba^{\core,m}=ba^{\core,m}aa^{\core,m}=aa^{\core,m}.
\end{align}
Using
$$
aa^{\core,m}=ba^{\core,m}aa^{\core,m}=bb^{\core,m}ba^{\core,m}aa^{\core,m}=bb^{\core,m}aa^{\core,m},
$$
we have
\begin{align*}
aa^{\core,m}&=m^{-1}\left(m aa^{\core,m} \right)^{*}=m^{-1}\left(mbb^{\core,m}aa^{\core,m} \right)
=m^{-1}\left(aa^{\core,m}\right)^{*}\left( mbb^{\core,m}\right)^{*}\\
&=m^{-1}\left(aa^{\core,m}\right)^{*}m bb^{\core,m}
=m^{-1}\left(aa^{\core,m}\right)^{*}m^{*} bb^{\core,m}
=m^{-1}\left(maa^{\core,m}\right)^{*} bb^{\core,m}\\
&=aa^{\core,m}bb^{\core,m}.
\end{align*}
Hence,
\begin{align}\label{425}
aa^{\core,m}=aa^{\core,m}bb^{\core,m}=bb^{\core,m}aa^{\core,m}.
\end{align}
Using the above relations, we obtain the following identities.
\begin{equation}\label{426}
bb^{\core,m}a^{\core,m}=bb^{\core,m}a\left(a^{\core,m}\right)^2=bb^{\core,m}ba^{\core,m}a\left(a^{\core,m}\right)^2=ba^{\core,m}a\left(a^{\core,m}\right)^2=a^{\core,m},
\end{equation}
\begin{equation}\label{427}
    a^{\core,m}bb^{\core,m}a^2=a^{\core,m}bb^{\core,m}ba^{\core,m}aa=a^{\core,m}ba^{\core,m}aa=a,
\end{equation}
\begin{center}
 $a^{\core,m}bb^{\core,m}ab=a^{\core,m}bb^{\core,m}ba^{\core,m}ab=a^{\core,m}ba^{\core,m}ab=a^{\core,m}ab$,  \end{center}
\begin{center}
   $a^{\core,m}ab^{\core,m}ba=a^{\core,m}ab^{\core,m}bba^{\core,m}a=a^{\core,m}a^2=a$, 
\end{center}
\begin{center}
 $b^{\core,m}ba=b^{\core,m}bba^{\core,m}a=ba^{\core,m}a=a,\quad b^{\core,m}a^2=b^{\core,m}ba^{\core,m}a^2=a$,   
\end{center}
\begin{center}
    $b^{\core,m}ab=b^{\core,m}ba^{\core,m}ab=b^{\core,m}ba\left(a^{\core,m}\right)^2ab=a^{\core,m}ab$,
\end{center}
\begin{equation}\label{432}
a^{\core,m}ab^{\core,m}ab=a^{\core,m}aa^{\core,m}ab=a^{\core,m}ab.
\end{equation}
Let $x=-\frac{1}{2}a^{\core,m}+b^{\core,m}+\frac{1}{2}a^{\core,m}bb^{\core,m}-\frac{1}{2}a^{\core,m}ab^{\core,m}.$
Using \eqref{424} and \eqref{425}, we obtain
\begin{align*}
(a+b)x&=(a+b)\left(-\dfrac{1}{2}a^{\core,m}+b^{\core,m}+\dfrac{1}{2}a^{\core,m}bb^{\core,m}-\dfrac{1}{2}a^{\core,m}ab^{\core,m} \right)\\
&=-\dfrac{1}{2}aa^{\core,m}+ab^{\core,m}+\dfrac{1}{2}aa^{\core,m}bb^{\core,m}-\dfrac{1}{2}ab^{\core,m}\\
&\quad -\dfrac{1}{2}ba^{\core,m}+bb^{\core,m}+\dfrac{1}{2}ba^{\core,m}bb^{\core,m}-\dfrac{1}{2}ba^{\core,m}ab^{\core,m}\\
&=-\dfrac{1}{2}aa^{\core,m}+\dfrac{1}{2}ab^{\core,m}+\dfrac{1}{2}aa^{\core,m}-\dfrac{1}{2}aa^{\core,m}+bb^{\core,m}
+\dfrac{1}{2}aa^{\core,m}bb^{\core,m}-\dfrac{1}{2}ab^{\core,m}\\
&=bb^{\core,m},
\end{align*}
which implies that
$$
(m(a+b)x)^{*}=(mbb^{\core,m})^{*}=mbb^{\core,m}=m(a+b)x.
$$
Using \eqref{426}, we have
\begin{align*}
(a+b)x^2&=bb^{\core,m}\left(-\dfrac{1}{2}a^{\core,m}+b^{\core,m}+\dfrac{1}{2}a^{\core,m}bb^{\core,m}-\dfrac{1}{2}a^{\core,m}ab^{\core,m} \right)\\
&=-\dfrac{1}{2}bb^{\core,m}a^{\core,m}+b^{\core,m}+\dfrac{1}{2}bb^{\core,m}a^{\core,m}bb^{\core,m}-\dfrac{1}{2}bb^{\core,m}a^{\core,m}ab^{\core,m} \\
& -\dfrac{1}{2}a^{\core,m}+b^{\core,m}+\dfrac{1}{2}a^{\core,m}bb^{\core,m}-\dfrac{1}{2}a^{\core,m}ab^{\core,m}\\
&=x.
\end{align*}
Using \eqref{427} and \eqref{432}, we obtain
\begin{align*}
x(a+b)^2&=\left(-\dfrac{1}{2}a^{\core,m}+b^{\core,m}+\dfrac{1}{2}a^{\core,m}bb^{\core,m}-\dfrac{1}{2}a^{\core,m}ab^{\core,m} \right)(a+b)^2\\
&=\left(-\dfrac{1}{2}a^{\core,m}+b^{\core,m}+\dfrac{1}{2}a^{\core,m}bb^{\core,m}-\dfrac{1}{2}a^{\core,m}ab^{\core,m} \right)(a^2+b^2+ab+ba)\\
&=-\dfrac{1}{2}a-\dfrac{1}{2}a^{\core,m}b^2-\dfrac{1}{2}a^{\core,m}ab-\dfrac{1}{2}a^{\core,m}ba+b^{\core,m}a^2+b+b^{\core,m}ab+b^{\core,m}ba\\
&\qquad+\dfrac{1}{2}a^{\core,m}bb^{\core,m}a^2+\dfrac{1}{2}a^{\core,m}b^2+\dfrac{1}{2}a^{\core,m}bb^{\core,m}ab+\dfrac{1}{2}a^{\core,m}bb^{\core,m}ba\\
&\qquad-\dfrac{1}{2}a^{\core,m}ab^{\core,m}a^2-\dfrac{1}{2}a^{\core,m}ab^{\core,m}b^2-\dfrac{1}{2}a^{\core,m}ab^{\core,m}ab-\dfrac{1}{2}a^{\core,m}ab^{\core,m}ba\\
&=-\dfrac{1}{2}a-\dfrac{1}{2}a^{\core,m}b^2-\dfrac{1}{2}a^{\core,m}ab--\dfrac{1}{2}a^{\core,m}ba+a+b+a^{\core,m}ab+a+\dfrac{1}{2}a+\dfrac{1}{2}a^{\core,m}b^2\\
&\qquad +\dfrac{1}{2}a^{\core,m}ab+\dfrac{1}{2}a^{\core,m}ba-\dfrac{1}{2}a-\dfrac{1}{2}a^{\core,m}ab-\dfrac{1}{2}a^{\core,m}ab-\dfrac{1}{2}a\\
&=a+b.
\end{align*}
Hence, by the Theorem \ref{core-equivalent}, we have
\begin{equation*}
(a+b)^{\core,m}=-\dfrac{1}{2}a^{\core,m}+b^{\core,m}+\dfrac{1}{2}a^{\core,m}bb^{\core,m}-\dfrac{1}{2}a^{\core,m}ab^{\core,m}.
\qedhere
\end{equation*}
\end{proof}
Next, we discus a characterization for difference of $m$-weighted core inverse of two elements in a ring.
\begin{theorem}
Let $a,b\in\mathcal{R}^{\core,m}$ and $m\in \mathcal{R}^{-1}$ be a hermitian element.
If $a=aa^{\core,m}b=ba^{\core,m}a,$ then $a-b\in \mathcal{R}^{\core,m}$  and
$$
(a-b)^{\core,m}= aa^{\core,m}b^{\core,m}-b^{\core,m}.
$$
\end{theorem}
\begin{proof}
Let $a=aa^{\core,m}b$. Then $ab^{\core,m}=aa^{\core,m}bb^{\core,m}$. From the proof of Theorem \ref{thm:sum-alternative}, we have
\begin{align*}
&aa^{\core,m}=aa^{\core,m}bb^{\core,m},~~b^{\core,m}ba=a=aa^{\core,m}b^{\core,m}ba=b^{\core,m}a^2,\\
&b^{\core,m}ab=a^{\core,m}ab=aa^{\core,m}b^{\core,m}ab,\\
&baa^{\core,m}b^{\core,m}=ba^{\core,m}a^2a^{\core,m}b^{\core,m}=a^2a^{\core,m}b^{\core,m}.
\end{align*}
Let $x=aa^{\core,m}b^{\core,m}-b^{\core,m}.$ Then
\begin{align*}
(a-b)x&=(a-b)\left(aa^{\core,m}b^{\core,m}-b^{\core,m}\right)\\
&=a^2a^{\core,m}b^{\core,m}-ab^{\core,m}-ba^{\core,m}b^{\core,m}+bb^{\core,m}\\
&=-aa^{\core,m}bb^{\core,m}+bb^{\core,m}\\
&=-aa^{\core,m}+bb^{\core,m},
\end{align*}
which shows that
\begin{align*}
\left(m(a-b)x \right)^*=-(maa^{\core,m})^*+(mbb^{\core,m})^*=-maa^{\core,m}+mbb^{\core,m}=m(a-b)x.
\end{align*}
Now,
\begin{align*}
(a-b)x^2&=(-aa^{\core,m}+bb^{\core,m})(aa^{\core,m}b^{\core,m}-b^{\core,m})\\
&=-aa^{\core,m}b^{\core,m}+aa^{\core,m}b^{\core,m}+bb^{\core,m}aa^{\core,m}b^{\core,m}-b^{\core,m}\\
&=aa^{\core,m}bb^{\core,m}b^{\core,m}-b^{\core,m}\\
&=aa^{\core,m}b^{\core,m}-b^{\core,m}=x.
\end{align*}
Again,
\begin{align*}
x(a-b)^2&=(aa^{\core,m}b^{\core,m}-b^{\core,m})(a^2+b^2-ab-ba)\\
&=aa^{\core,m}b^{\core,m}a^2+aa^{\core,m}b-aa^{\core,m}b^{\core,m}ab-aa^{\core,m}b^{\core,m}ba-b^{\core,m}a^2-b\\
&\qquad+b^{\core,m}ab+b^{\core,m}ba\\
&=a+aa^{\core,m}b-a^{\core,m}ab-a-a-b+a^{\core,m}ab+a\\
&=a-b.
\end{align*}
Finally, with the help of Theorem \ref{core-equivalent}, the required result can be established. 
\end{proof}

\section{Reverse order law}
In this this section, we will discuss the existence of reverse order order law for both $m$-weighted core and $n$-weighted dual core inverses. In addition, a few mixed-type reverse order law are discussed. The very first result contain two sufficient conditions for the reverse order law for weighted core inverse. 

\begin{theorem}\label{thm4.1}
Let $ a,b\in\mathcal{R}^{\core, m}$ and $m\in\mathcal{R}^{-1}$  with $m^*=m$. If
\begin{equation*}
    a^{\core,m}b=b^{\core,m}a~~~~\textnormal{and}~~~ aa^{\core,m}=ba^{\core,m},
\end{equation*}
 then
\begin{equation*}
\left(ab\right)^{\core,m}=b^{\core,m}a^{\core,m}=\left(a^{\core,m}\right)^2=(a^2)^{\core,m}.
\end{equation*}
\end{theorem}
\begin{proof}
Let $a^{\core,m}b=b^{\core,m}a$. Then
\begin{center}
 $b^{\core,m}a^{\core,m}=(b^{\core,m}a)(a^{\core,m})^2
= a^{\core,m}ba^{\core,m}a^{\core,m}
= a^{\core,m}aa^{\core,m}a^{\core,m}
= (a^{\core,m})^2$.
\end{center}
From Lemma \ref{lm:a-n-core}, we get $b^{\core,m}a^{\core,m}= (a^{\core,m})^2
= (a^2)^{\core,m}$. Now we will verify that, $b^{\core,m}a^{\core,m}$ is the $m$-weighted core inverse of $ab$. Let $x=b^{\core,m}a^{\core,m}$. Then
\begin{eqnarray*}
  (mabx)^*&=&(mabb^{\core,m}a^{\core,m})^*=\left(ma(ba^{\core,m})a^{\core,m}\right)^*=(maaa^{\core,m}a^{\core,m})^*\\
  &=&(maa^{\core,m})^*=maa^{\core,m}=mabx,
\end{eqnarray*}
\begin{eqnarray*}
x(ab)^2&=&b^{\core,m}a^{\core,m}(ab)^2=(a^{\core,m})^2ab(a)b=(a^{\core,m})^2a(ba^{\core,m})a^2b\\
&=&(a^{\core,m})^2a^2a^{\core,m}a^2b=(a^{\core,m})^2a^3b=ab,\mbox{ and }
\end{eqnarray*}
\begin{eqnarray*}
abx^2&=&ab(b^{\core,m}a^{\core,m})^2=ab(a^{\core,m})^4=a(ba^{\core,m})(a^{\core,m})^3=a^2(a^{\core,m})^4=(a^{\core,m})^2\\
&=&b^{\core,m}a^{\core,m}=x.
\end{eqnarray*}
Therefore, $(ab)^{\core,m}=b^{\core,m}a^{\core,m}$.
\end{proof}

\begin{corollary}\label{thm44.1}
Let $ a,b\in\mathcal{R}^{\core, m}$. Assume that $m\in\mathcal{R}^{-1}$ is a hermitian element. If
$a^{\core,m}b=b^{\core,m}a$ and $aa^{\core,m}=ba^{\core,m}$, then
$\left(ab\right)^{\core,m}=b^{\core,m}a^{\core,m}$.
\end{corollary}
The converse of the above corollary is not true in general as shown in the next example. 
\begin{example}\rm
Let $R=M_2(\mathbb{R})$, $a=\begin{pmatrix}
1  & 0\\
0 & 0
\end{pmatrix}$, $m=\begin{pmatrix}
1  & 0\\
0 & 2
\end{pmatrix}$ and  $b=\begin{pmatrix}
-1  & 1\\
0 & 0
\end{pmatrix}$. \\
Using Theorem \ref{core-equivalent} (iv), one can find $a^{\core,m}=\begin{pmatrix}
1 & 0\\
0 & 0
\end{pmatrix}$, $b^{\core,m}=\begin{pmatrix}
-1 & 0\\
0 & 0
\end{pmatrix}=(ab)^{\core,m}$.\\
Thus $(ab)^{\core,m}=b^{\core,m}a^{\core,m}$. However,
\begin{equation*}
a^{\core,m}b=\begin{pmatrix}
-1 & 1\\
0 & 0
\end{pmatrix}\neq \begin{pmatrix}
-1 & 0\\
0 & 0
\end{pmatrix}=b^{\core,m}a.
\end{equation*}
\end{example}

The following result can be verified for $n$-weighted dual core inverse. 
\begin{theorem}\label{thm4.2}
Let $n\in \mathcal{R}^{-1}$ be a hermitian element and $ a,b \in\mathcal{R}_{n,\core}$. If
\begin{equation*}
ab_{n,\core}=ba_{n,\core} ~~~~\textnormal{and}~~~~ b_{n,\core}b=b_{n,\core}a,
\end{equation*}
then
\begin{equation*}
(ab)_{n,\core}=b_{n,\core}a_{n,\core}=(b_{n,\core})^2=(b^2)_{n,\core}.
\end{equation*}
\end{theorem}
In view of the Theorem \ref{thm4.2} we can state the following result as a corollary. 
\begin{corollary}\label{colonn}
Let $n\in \mathcal{R}^{-1}$ be a hermitian element and $ a,b \in\mathcal{R}_{n,\core}$. If
\begin{equation*}
ab_{n,\core}=ba_{n,\core} ~~~~\textnormal{and}~~~~ b_{n,\core}b=b_{n,\core}a,
\end{equation*}
then
\begin{equation*}
(ab)_{n,\core}=b_{n,\core}a_{n,\core}.
\end{equation*}
\end{corollary}
The converse of the above corollary is not true in general as shown in the following example.
\begin{example}\rm
Let $R=M_2(\mathbb{R})$, $a=\begin{pmatrix}
1  & 0\\
-1 & 0
\end{pmatrix}$, $n=\begin{pmatrix}
1  & 0\\
0 & 2
\end{pmatrix}$ and  $b=\begin{pmatrix}
-1  & 0\\
1 & 0
\end{pmatrix}$. \\
Using Theorem \ref{n-core-equivalent} (iv), one can find $a_{n,\core}=\begin{pmatrix}
1 & 0\\
0 & 0
\end{pmatrix}$, $b_{n,\core}=\begin{pmatrix}
-1 & 0\\
0 & 0
\end{pmatrix}=(ab)_{n,\core}$.\\
Thus $(ab)^{\core,m}=b^{\core,m}a^{\core,m}$. However,
\begin{equation*}
b_{n,\core}b=\begin{pmatrix}
1 & 0\\
0 & 0
\end{pmatrix}\neq \begin{pmatrix}
-1 & 0\\
0 & 0
\end{pmatrix}=b_{n,\core,m}a.
\end{equation*}
\end{example}

A few necessary condition for the reverse order law  are established in the following theorems.  
\begin{theorem}\label{thm4.3}
Let $m\in\mathcal{R}^{-1}$ be a hermitian element and $ a,b \in \mathcal{R}^{\core, m}$. If
$(ab)^{\core,m}=b^{\core,m}a^{\core,m}$, then 
\begin{enumerate}[(i)]
   \item $ab=bb^{\core,m}ab=b^{\core,m}bab$;
    \item  $b^{\core,m}a \mathcal{R}\subseteq ab\mathcal{R} \subseteq ba\mathcal{R}$;
    \item  $bb^{\core, m}a^{\core, m}\in  c\{3^m,6\}$, where $c= abb^{\core, m}$.
\end{enumerate}

\end{theorem}
\begin{proof}
\begin{enumerate}[(i)]
\item Let, $(ab)^{\core,m}=b^{\core,m}a^{\core,m}$. Then
\begin{align*}
ab&= (ab)^{\core,m}(ab)^2=b^{\core,m}a^{\core,m}(ab)^2=b(b^{\core,m})^2a^{\core,m}(ab)^2\\
&= bb^{\core,m}(b^{\core,m}a^{\core,m})(ab)^2= bb^{\core,m}(ab)^{\core,m}(ab)^2=bb^{\core,m}ab.
\end{align*}
and
\begin{align*}
ab&= (ab)^{\core,m}(ab)^2=b^{\core,m}b\left(b^{\core,m}a^{\core,m} (ab)^2\right)=b^{\core,m}bab.
\end{align*}

\item From (i), we obtain
\begin{eqnarray*}
ab=bb^{\core,m}ab = bb^{\core,m}a^{\core,m}a^2b= b(ab)^{\core,m}a^2b=bab((ab)^{\core,m})^2a^2b.
\end{eqnarray*}
Thus $ab\mathcal{R} \subseteq  ba\mathcal{R}$. Further from
\begin{center}
   $b^{\core,m}a=b^{\core,m}a^{\core,m}a^2=(ab)^{\core,m}a^2=ab((ab)^{\core,m})^2a^2$, we obtain
\end{center} $ \left(b^{\core,m}a\right)\mathcal{R}\subseteq (ab)\mathcal{R}$.

\item $ {b}  {b}^{\core, m}  {a}^{\core, m}\in {c}\{3^m\}$ is follows from the following
 \begin{eqnarray*}
 (mcbb^{\core, m}a^{\core, m})^*&=& (mabb^{\core, m}  {b}  {b}^{\core, m}  {a}^{\core, m})^*= (m{a}  {b}  {b}^{\core, m}  {a}^{\core, m})^*\\
 &=& (m{a}  {b} ( {a}  {b})^{\core, m})^*=m{a}  {b} ( {a}  {b})^{\core, m}= m{c}  {b}  {b}^{\core, m}  {a}^{\core, m}.
 \end{eqnarray*}

 Using $ {a}  {b}= {b}  {b}^{\core, m} {a}  {b}$ from part $(a)$, we obtain
 \begin{eqnarray*}
  {b} {b}^{\core,m} {a}^{\core,m} {c}^2&=&  {b} {b}^{\core,m} {a}^{\core,m} {a} ({b} {b}^{\core,m} {a} {b}){b}^{\core,m}\\
 &=& {b} {b}^{\core,m} {a}^{\core,m} {a}^2 {b} {b}^{\core,m}= {b} {b}^{\core,m} {a} {b} {b}^{\core,m}= {a} {b} {b}^{\core,m}.
 \end{eqnarray*}
 Hence $ {b} {b}^{\core,m} {a}^{\core,m}\in {c}\{6\}$ and completes the proof.
\qedhere
 \end{enumerate}
 \end{proof}
We now discuss a sufficient condition for $m$-weighted core inverse. 
\begin{theorem}
Let $ a,b \in \mathcal{R}^{\core, m}$ and $m\in\mathcal{R}^{-1}$ with $m=m^*$.  If $a^2=ba$, then
\begin{enumerate}[(i)]
\item  $ab\in\mathcal{R}^{\core,m}$   and  $(ab)^{\core,m}=b^{\core,m}a^{\core,m};$
\item $abb^{\core,m}\in\mathcal{R}^{\core,m}$ and $\left(abb^{\core,m}\right)^{\core,m}=bb^{\core,m}a^{\core,m}.$
\end{enumerate}
\end{theorem}

\begin{proof}
\begin{enumerate}[(i)]
\item
Using $a^2=ba$ repetitively, we get
\begin{equation}\label{eqn4.1}
a=a^2a^{\#}=baa^{\#}=b^{\core,m}b^2aa^{\#}=b^{\core,m}ba^2a^{\#}=b^{\core,m}a^3a^{\#}=b^{\core,m}a^2,
\end{equation}
and
\begin{equation}\label{eqq4.2}
    aa^{\core,m}=a^2(a^{\core,m})^2=ba(a^{\core,m})^2=ba^{\core,m}.
\end{equation}
Applying Equation \eqref{eqn4.1}, we get
\begin{equation}\label{eqn4.2}
    a^{\core,m}=a(a^{\core,m})^2=b^{\core,m}a^2(a^{\core,m})^2=b^{\core,m}ba(a^{\core,m})^2=b^{\core,m}ba^{\core,m},
\end{equation}
 and
\begin{equation*}
b^{\core,m}a^{\core,m}(ab)^2=b^{\core,m}a^{\core,m}a(ba)b=b^{\core,m}a^{\core,m}a^3b=b^{\core,m}a^2b=ab.
\end{equation*}
From Equation \eqref{eqn4.2}, we have
\begin{eqnarray*}
\nonumber
   (mabb^{\core,m}a^{\core,m})^*&=&(mabb^{\core,m}b^{\core,m}ba^{\core,m})^*=(mab^{\core,m}ba^{\core,m})^*\\
   &=&(maa^{\core,m})^*=maa^{\core,m}=mabb^{\core,m}a^{\core,m}.
\end{eqnarray*}
Further, from Equations \eqref{eqn4.1} and \eqref{eqn4.2}, we obtain
\begin{eqnarray*}\label{eqn4.5}
\nonumber
ab(b^{\core,m}a^{\core,m})^2&=&abb^{\core,m}a^{\core,m}b^{\core,m}a^2(a^{\core,m})^3=abb^{\core,m}a^{\core,m}a(a^{\core,m})^3\\
\nonumber
&=&abb^{\core,m}(a^{\core,m})^3=abb^{\core,m}b^{\core,m}ba^{\core,m}(a^{\core,m})^2\\
\nonumber
&=&ab^{\core,m}ba^{\core,m}(a^{\core,m})^2=a(a^{\core,m})^3\\
&=&b^{\core,m}a^2(a^{\core,m})^3=b^{\core,m}a^{\core,m}.
\end{eqnarray*}
Thus using Lemma \ref{lm:m-weighted-core-alt}, we conclude that $b^{\core,m}a^{\core,m}$ is the $m$-weighted core inverse of $ab$,
which proves first part of the theorem.

\item Applying equation \eqref{eqq4.2} and \eqref{eqn4.2}, we have $a^{\core,m}=b^{\core,m}ba^{\core,m}$ and $ba^{\core,m}=aa^{\core,m}$, which shows that
$$
bb^{\core,m}a^{\core,m}=bb^{\core,m}\left(ba^{\core,m}\right)=b^{\core,m}aa^{\core,m}.
$$
Therefore,
\begin{align*}
bb^{\core,m}a^{\core,m}\left(abb^{\core,m}\right)^2
&=b^{\core,m}\left(aa^{\core,m}a\right)\left(bb^{\core,m}ab\right)b^{\core,m}
=b^{\core,m}aabb^{\core,m}=\left(b^{\core,m}bab\right)b^{\core,m}\\
&=abb^{\core,m},
\end{align*}
which shows that $bb^{\core,m}a^{\core,m} \in abb^{\core,m}\{7\}.$
Using part (i) of this theorem and part (iii) of Theorem \ref{thm4.3}, we have $bb^{\core,m}a^{\core,m} \in abb^{\core,m}\{3^m,6\}.$ Finally, with the help of Theorem \ref{core-equivalent}, we obtain the required result.
\qedhere
\end{enumerate}
\end{proof}

One can prove the following result for $n$-weighted dual core inverse. 

\begin{theorem}
Let $ a,b \in\mathcal{R}_{n,\core}$ and $n\in\mathcal{R}^{-1}$ with $n=n^*$. If $b^2=ba$, then
$ab\in \mathcal{R}_{n,\core}$  and $(ab)_{n,\core}=b_{n,\core}a_{n,\core}$.
\end{theorem}

A few necessary and sufficient condition for $m$-weighted core inverse are presented below. 
\begin{theorem}
Let $m\in\mathcal{R}^{-1}$ be a hermitian element and $ a,b,ab\in \mathcal{R}^{\core, m}$ and $a^*mb\mathcal{R}=mba^*\mathcal{R}$.
Then
\begin{equation*}
(ab)^{\core,m}=b^{\core,m}a^{\core,m}
\end{equation*}
 if and only if
\begin{enumerate}[(i)]
    \item $(b^{\core,m}a)\mathcal{R}\subseteq (ab) \mathcal{R} \subseteq  (ba) \mathcal{R}$
    \item $mbb^{\core,m}aa^{\core,m}=maa^{\core,m}bb^{\core,m}$ or $bb^{\core,m}aa^{\core,m}=aa^{\core,m}bb^{\core,m}$
\end{enumerate}
\end{theorem}

\begin{proof}
Let $(ab)^{\core,m}=b^{\core,m}a^{\core,m}$. Then the result (i) follows from Theorem \ref{thm4.3}. Further,
\begin{equation}\label{eq4.5}
    ab=b^{\core,m}a^{\core,m}(ab)^2=b(b^{\core,m})^2a^{\core,m}(ab)^2=bb^{\core,m}ab.
\end{equation}
Using the relation of $(a^*mb)\mathcal{R}=(mba^*)\mathcal{R}$, we obtain $b^*ma=uab^*m$ for some $u\in \mathcal{R}$.
In addition,
\begin{eqnarray}\label{eq4.6}
\nonumber
b^*ma&=&ua(mb)^*=ua(mbb^{\core,m}b)^*=uab^*(mbb^{\core,m})^*= uab^*mbb^{\core,m}\\
&=&b^*mabb^{\core,m}.
\end{eqnarray}
Using $(ab)\mathcal{R} \subseteq (ba)\mathcal{R}$, equation\eqref{eq4.5}, and equation \eqref{eq4.6}, we obtain
\begin{eqnarray}\label{eq4.60}
\nonumber
mbb^{\core,m}aa^{\core,m}&=& (mbb^{\core,m})^*aa^{\core,m}= (b^{\core,m})^*b^*maa^{\core,m}\\
\nonumber
&=& (b^{\core,m})^*b^*mabb^{\core,m}a^{\core,m}=(mbb^{\core,m})^*abb^{\core,m}a^{\core,m}\\
&=& mbb^{\core,m}abb^{\core,m}a^{\core,m}=mabb^{\core,m}a^{\core,m}\\
\nonumber
&=& (mabb^{\core,m}a^{\core,m})^*=(mbb^{\core,m}m^{-1}maa^{\core,m})^*\\
\nonumber
&=& (maa^{\core,m})^*m^{-1}(mbb^{\core,m})^*=maa^{\core,m}bb^{\core,m}.
\end{eqnarray}

 Conversely, let $({a}{b})\mathcal{R}\subseteq ({b}{a})\mathcal{R}$. This yields ${a}{b}={b}{a}{u}$ for some ${u}\in\mc$ and
 \begin{equation}\label{eqn4.7}
     {a}{b}={b}{a}{u}={b}{b}^{\core,m}{b}{a}{u}={b}{b}^{\core,m}{a}{b}.
 \end{equation}
 Using equation \eqref{eqn4.7} along with $ ({a}^*m{b})\mathcal{R}=(m{b}{a}^*)\mathcal{R}$, we can easily prove (like equation \eqref{eq4.60} )
 \begin{equation}\label{eq4.8}
   m{b}{b}^{\core,m}{a}{a}^{\core,m}=m{a}{b}{b}^{\core,m}{a}^{\core,m} \mbox{ or }  {b}{b}^{\core,m}{a}{a}^{\core,m}={a}{b}{b}^{\core,m}{a}^{\core,m}
 \end{equation}
 Combining $mbb^{\core,m}aa^{\core,m}=maa^{\core,m}bb^{\core,m}$ and equation\eqref{eq4.8}, we get
 \begin{center}
$(m{a}{b}{b}^{\core,m}{a}^{\core,m})^*=(m{b}{b}^{\core,m}{a}{a}^{\core,m})^*=ma{a}^{\core,m}{b}{b}^{\core,m}=ma{b}{b}^{\core,m}a^{\core,m}$ .
 \end{center}
 From $ ({b}^{\core}{a})\mathcal{R}\subseteq ({a}{b})\mathcal{R}$ and second part of equation \eqref{eq4.8}, we obtain
\begin{eqnarray*}
{a}{b}({b}^{\core,m}{a}^{\core,m})^2&=&{b}{b}^{\core,m}{a}{a}^{\core,m}{b}^{\core,m}{a}^{\core,m}={a}{a}^{\core,m}{b}{b}^{\core,m}{b}^{\core,m}{a}^{\core,m}\\
&=&{a}{a}^{\core,m}({b}^{\core,m}{a})({a}^{\core,m})^2={a}{a}^{\core,m}({a}{b}{v})({a}^{\core,m})^2\\
&=&({a}{b}{v})({a}^{\core,m})^2=({b}^{\core,m}{a})({a}^{\core,m})^2={b}^{\core,m}{a}^{\core,m}.
\end{eqnarray*}
Further,  ${a}{b}={b}{b}^{\core,m}{a}{b}={b}{b}^{\core,m}{a}{a}^{\core,m}{a}{b}={a}{b}{b}^{\core,m}{a}^{\core,m}{a}{b}$.
Thus by  Lemma \ref{lm:m-weighted-core-alt}, ${b}^{\core,m}{a}^{\core,m}$ is the $m$-weighted core inverse of $ab$.
\end{proof}

\begin{theorem}
Let $a,b,ab\in\mathcal{R}^{\core,m}$ and $m\in\mathcal{R}^{-1}$ be a hermitian element. Then 
\begin{equation*}
 (ab)^{\core,m}=b^{\core,m}a^{\core,m}    
\end{equation*}
if and only if 
\begin{equation*}
 b(ab)^{\core,m}=bb^{\core,m}a^{\core,m} \textnormal{~~and~~}
 abb^{\core,m}=b^{\core,m}babb^{\core,m}.
  \end{equation*}
\end{theorem}
\begin{proof}
Necessary part of this theorem can be proved using Theorem \ref{thm4.3}. Now will prove the sufficient part. For this, let us assume that
$b(ab)^{\core,m}=bb^{\core,m}a^{\core,m}$ and  $abb^{\core,m}=b^{\core,m}babb^{\core,m}$. Then
$$
a\left(bb^{\core,m}a^{\core,m}\right)ab=ab(ab)^{\core,m}ab=ab \quad \mbox{ and }$$
$$(mabb^{\core,m}a^{\core,m})^*=(mab(ab)^{\core,m})^*=mab(ab)^{\core,m}=mabb^{\core,m}a^{\core,m}.
$$
This leads $b^{\core,m}a^{\core,m}\in ab\{1,3^m\}.$ Furthermore,
\begin{align*}
b^{\core,m}a^{\core,m}&=b^{\core,m}\left(bb^{\core,m}a^{\core,m}\right)=b^{\core,m}b(ab)^{\core,m}=b^{\core,m}bab\left((ab)^{\core,m} \right)^2\\
&=\left(b^{\core,m}babb^{\core,m}\right)b \left((ab)^{\core,m}\right)^2 = abb^{\core,m}b\left((ab)^{\core,m}\right)^2,
\end{align*}
which implies that $b^{\core,m}a^{\core,m}\mathcal{R}\subseteq ab\mathcal{R}.$ Finally, using Theorem \ref{thm:m-core-eqv}, we have
\begin{equation*}
(ab)^{\core,m}=b^{\core,m} a^{\core,m}.
\qedhere
\end{equation*}
\end{proof}

\begin{theorem}\label{th:cor-dedek}
Let $a,b\in \mathcal{R}^{\core,m}$ and $m\in\mathcal{R}^{-1}$ with $m=m^*$. Then the following statements are equivalent:
\begin{enumerate}[(i)]
\item $ab,ba\in \mathcal{R}^{\core,m}$ with $a^{\core,m}b^{\core,m}=(ba)^{\core,m},$ $b^{\core,m}a^{\core,m}=(ab)^{\core,m}.$ 
\item $baa^{\core,m},abb^{\core,m}\in\mathcal{R}^{\core,m}$ with $bb^{\core,m}a^{\core,m}=(abb^{\core,m})^{\core,m}.$
$aa^{\core,m}b^{\core,m}=\left(baa^{\core,m}\right)^{\core,m},$ $ba\mathcal{R}=(ba)^2\mathcal{R}$ and $ab\mathcal{R}=(ab)^2\mathcal{R}.$
\end{enumerate}
\end{theorem}
\begin{proof}
(i) $\Rightarrow$ (ii)\hspace{0.3cm} Clearly, $(ab)^2\mathcal{R}=ab\mathcal{R}$ and $(ba)^2\mathcal{R}=ba\mathcal{R}.$ 
Using symmetry, it is enough to
show that $aba^{\core,m}\in\mathcal{R}^{\core,m}$ with $\left(abb^{\core,m}\right)^{\core,m}=bb^{\core,m}a^{\core,m}.$ Using Theorem \ref{thm4.3}, we obtain $bb^{\core,m}a^{\core,m}\in abb^{\core,m}\{3^m,6\}$ and $ba\mathcal{R}=ab\mathcal{R}.$ From the range condition $ba\mathcal{R}=ab\mathcal{R}$, we have
 $ba=abt$ for some $t\in\mathcal{R}.$ Now
\begin{align*}
abb^{\core,m}\left(bb^{\core,m}a^{\core,m}\right)^2
&=abb^{\core,m}a^{\core,m}bb^{\core,m}a^{\core,m}=ab(ab)^{\core,m}b(ab)^{\core,m}\\
&=ab(ab)^{\core,m}bab\left((ab)^{\core,m}\right)^2 =\left(ab(ab)^{\core,m}ab \right) ub\left((ab)^{\core,m} \right)^2\\
&=abub\left((ab)^{\core,m}\right)^2=bab\left((ab)^{\core,m}\right)^2=b(ab)^{\core,m}\\
&=bb^{\core,m}a^{\core,m}.
\end{align*}
Using Theorem \ref{core-equivalent}, we obtain $abb^{\core,m}\in \mathcal{R}^{\core,m}.$
with $\left(abb^{\core,m}\right)^{\core,m}=bb^{\core,m}a^{\core,m}.$

\noindent(ii) $\Rightarrow$ (i)\hspace{0.3cm} From given hypothesis $(abb^{\core,m})^{\core,m}=bb^{\core,m}a^{\core,m},$ we obtain
\begin{align*}
b^{\core,m}a^{\core,m}abb^{\core,m}a^{\core,m}
=b^{\core,m}(bb^{\core,m}a^{\core,m}abb^{\core,m}bb^{\core,m}a^{\core,m})=b^{\core,m}bb^{\core,m}a^{\core,m}=b^{\core,m}a^{\core,m}
\end{align*}
and
$$
abb^{\core,m}a^{\core,m}=abb^{\core,m}bb^{\core,m}a^{\core,m}=abb^{\core,m}\left(abb^{\core,m}\right)^{\core,m},
$$
which yield $b^{\core,m}a^{\core,m}\in ab\{2,3^m\}$. Furthermore,  we have
\begin{equation}\label{eq333}
abb^{\core,m}=bb^{\core,m}a^{\core,m}\left( abb^{\core,m}\right)^2
=b^{\core,m}b\left(bb^{\core,m}a^{\core,m}\left( abb^{\core,m}\right)^2\right)=b^{\core,m}babb^{\core,m}.
\end{equation}
Using equation \eqref{eq333} $bb^{\core,m}b=b^{\core,m}babb^{\core,m}b$. Thus $ ab=b^{\core,m}bab$.
Similarly, we can derive that $a^{\core,m}aba=ba.$ Therefore,
$$
ab=b^{\core,m}(ba)b=b^{\core,m}a^{\core,m}(ab)^2,
$$
which implies $ab\mathcal{R}\subseteq b^{\core,m}a^{\core,m}\mathcal{R}$ and $\mathcal{R}ab=\mathcal{R}(ab)^2.$
It can be noted from Lemma \ref{lm:group-inv} and $ab\mathcal{R}=(ab)^2\mathcal{R},$ that $ab\in \mathcal{R}^{\#}.$
Further, using Theorem \ref{thm:m-core-eqv}, we claim that $ab\in \mathcal{R}^{\core,m}$ with $(ab)^{\core,m}=b^{\core,m}a^{\core,m}.$ Similarly, we can  prove $ba\in \mathcal{R}^{\core,m}$  and  $(ba)^{\core,m}=a^{\core,m}b^{\core,m}$.
\end{proof}
We now present one result based on Dedekind-finite ring. 
\begin{corollary}
Suppose that $\mathcal{R}$ is a Dedekind-finite ring. If $m\in\mathcal{R}^{-1}$ is a hermitian element and $a,b\in\mathcal{R}^{\core,m},$ 
then the following  are equivalent. 
\begin{enumerate}[(i)]
    \item  $ba,ab\in \mathcal{R}^{\core,m}$ and $b^{\core,m}a^{\core,m}=(ab)^{\core,m},$ $a^{\core,m}b^{\core,m}=(ba)^{\core,m}.$
\item $baa^{\core,m},abb^{\core,m}\in\mathcal{R}^{\core,m}$ and $\left(baa^{\core,m}\right)^{\core,m}=aa^{\core,m}b^{\core,m},$ 
    $\left(abb^{\core,m}\right)^{\core,m}=bb^{\core,m}a^{\core,m}.$
\end{enumerate}
\end{corollary}
\begin{proof}
In the view of Theorem \ref{th:cor-dedek}, it is enough to show that $ba$ and $ab$ are regular. 
Let us assume that $bb^{\core,m}a^{\core,m}=\left(abb^{\core,m}\right)^{\core,m}.$ Therefore,
\begin{align*}
bb^{\core,m}abb^{\core,m}=bb^{\core,m}\left(bb^{\core,m}a^{\core,m}\left(abb^{\core,m}\right)^2\right)
=bb^{\core,m}a^{\core,m}\left(abb^{\core,m}\right)^2=abb^{\core,m}.
\end{align*}
This implies 
\begin{align*}
ab=\left(abb^{\core,m} \right)b=bb^{\core,m}abb^{\core,m}b=bb^{\core,m}ab=\left(bb^{\core,m}a^{\core,m}\right)a^2b
=abb^{\core,m}\left(bb^{\core,m}a^{\core,m} \right)^2aab. 
\end{align*}
Hence, $ab$ is regular. Using the hypothesis $aa^{\core,m}b^{\core,m}=\left(baa^{\core,m}\right)^{\core,m}$, one can derive that 
$ba$ is regular in a similar way. 
\end{proof}
To prove an equivalent characterization for mixed-type reverse order law, we use the Cline's formula \cite{Cline-1965}, as follows. 

\begin{lemma}\cite{Cline-1965}\label{lm:Darzin-inv}
Let $x,y\in\mathcal{R}.$ If $xy$ is Darzin invertible, then $yx$ is also Darzin invertible and $(yx)^D=y\left((xy)^D\right)^2x.$
\end{lemma}

\begin{theorem}
Let $m\in\mathcal{R}^{-1}$ be a hermitian element and $a,b\in \mathcal{R}^{\core,m}.$ Then the following are equivalent:
\begin{enumerate}[(i)]
  \item $a^{\core,m}ab=baa^{\core,m},$ $b^{\core,m}ba=abb^{\core,m}$ and $aa^{\core,m}b^{\core,m}=b^{\core,m}a^{\core,m}a.$
   \item $ab,ba\in \mathcal{R}^{\#}$ with $(ab)^{\#}=b^{\core,m}a^{\core,m},$ $(ba)^{\#}=a^{\core,m}b^{\core,m}$.
\end{enumerate}
\end{theorem}
\begin{proof}
(i) $\Rightarrow$ (ii)\hspace{0.3cm} Using the given hypothesis, we obtain
\begin{align*}
abb^{\core,m}(a^{\core,m}ab)=abb^{\core,m}baa^{\core,m}=a\left(baa^{\core,m}\right)=aa^{\core,m}ab=ab,
\end{align*}
\begin{align*}
\left(b^{\core,m}a^{\core,m}a\right)bb^{\core,m}a^{\core,m}&=aa^{\core,m}b^{\core,m}bb^{\core,m}a^{\core,m}
=\left(aa^{\core,m}b^{\core,m}\right)a^{\core,m}\\
&=b^{\core,m}a^{\core,m}aa^{\core,m}=b^{\core,m}a^{\core,m},
\end{align*}
and
$$
b^{\core,m}\left(a^{\core,m}ab\right)=\left(b^{\core,m}ba\right)a^{\core,m}=abb^{\core,m}a^{\core,m}.
$$
It follows that  $b^{\core,m}a^{\core,m}\in ab\{1,2,5\}.$ Hence, $(ab)^{\#}=b^{\core,m}a^{\core,m}.$ Similarly, we can prove that
$(ba)^{\#}=a^{\core,m}b^{\core,m}.$

\noindent (ii) $\Rightarrow$ (i)\hspace{0.3cm} From the given hypothesis $(ab)^{\#}=b^{\core,m}a^{\core,m},$ we obtain
\begin{align*}
ab&=b^{\core,m}a^{\core,m}(ab)^2=b^{\core,m}b  \left(b^{\core,m}a^{\core,m}(ab)^2\right)=b^{\core,m}bab\quad \mbox{ and}\\
ab&=(ab)^2b^{\core,m}a^{\core,m}=\left((ab)^2b^{\core,m}a^{\core,m} \right)aa^{\core,m}=abaa^{\core,m}.
\end{align*}
Similarly, using $(ba)^{\#}=a^{\core,m}b^{\core,m},$ we have $ba=a^{\core,m}aba=babb^{\core,m}.$
Therefore,
$$
a^{\core,m}ab=\left(a^{\core,m}aba \right)a^{\core,m}=baa^{\core,m} \mbox{ and } b^{\core,m}ba=\left(b^{\core,m}bab \right) b^{\core,m}=abb^{\core,m}.
$$
Since $ab$, $ba\in \mathcal{R}^{\#}$, using Lemma \ref{lm:Darzin-inv}, we have
$a(ba)^{\#}=ab\left((ab)^{\#} \right)^2=a=(ab)^{\#}a$. Which implies that $aa^{\core,m} b^{\core,m}= b^{\core,m} a^{\core,m}a.$
\end{proof}

The last result of reverse order law uses unitary condition of $m$-weighted core inverse.

\begin{theorem}\label{unit}
Let $m\in\mathcal{R}^{-1}$ be a hermitian element and $ a,b,ab \in \mathcal{R}^{\core, m}$.
\begin{enumerate}[(i)]
    \item If the element $ {b}$ is unitary  and
    ${b}^*  {a}^{\core, m}\mathcal{R} \subseteq  {a}^{\core, m} \mathcal{R}$, then
    \begin{equation*}
        ( {a}  {b})^{\core, m}= {b}^*  {a}^{\core, m}.
    \end{equation*}
    \item  If the element $ {a}$ is unitary and $a\mathcal{R}\subseteq b \mathcal{R}$, then
    \begin{equation*}
    ( {a}  {b})^{\core, m}= {b}^{\core, m}  {a}^{*}.
    \end{equation*}
\end{enumerate}
\end{theorem}

\begin{proof}
\begin{enumerate}[(i)]
\item Let ${b}^*  {a}^{\core, m}\mathcal{R}\subseteq  {a}^{\core, m}\mathcal{R}$. Then $ {b}^*  {a}^{\core, m}= {a}^{\core, m}  {u}$ for some $ {u}\in \mathcal{R}$. Now
\begin{center}
   $ {a}  {b}  {b}^*  {a}^{\core, m}  {a}  {b}= {a}  {a}^{\core, m}  {a}  {b}= {a}  {b}$,  \\
   $( m{a}  {b}  {b}^*  {a}^{\core, m})^*=( m{a}  {a}^{\core, m})^*= m{a}  {a}^{\core, m}= m{a}  {b}  {b}^*  {a}^{\core,m}$, and
\end{center}
\begin{eqnarray*}
 {a}  {b} ( {b}^*  {a}^{\core, m})^2= {a}  {a}^{\core, m}  {b}^*  {a}^{\core, m}= {a}  {a}^{\core, m}  {a}^{\core, m}  {u}= {a}^{\core, m}  {u}= {b}^*  {a}^{\core, m}.
\end{eqnarray*}
Thus by Lemma \ref{lm:m-weighted-core-alt}, $( {a}  {b})^{\core, m}= {b}^*  {a}^{\core, m}$.

\item Let $ a \mathcal{R}\subseteq b \mathcal{R}$. Then $ {a}= {b}  {u}$ for some $ {u}\in \mathcal{R}$ and
\begin{center}
    $ {b}^{\core, m}  {a}^{*} (  {a}  {b})^2= {b}^{\core, m}  {b}  {a}  {b}= {b}^{\core, m}  {b}^2u {b}={b}  {u}  {b}= {a}  {b}$.
\end{center}
Applying Lemma \ref{pro2.2}, we have
\begin{center}
  $ {b}^{\core, m}  {a}^{*}  {a}  {b}  {b}^{\core, m}  {a}^{*}= {b}^{\core, m}   {b}  {b}^{\core, m}  {a}^{*}= {b}^{\core, m}  {a}^{*}$.
\end{center}
Further,
\begin{eqnarray*}
(mab{b}^{\core, m}a^*)^*&=& (aa^*mabb^{\core, m}a^{*})^*=abb^{\core, m}a^*maa*\\
&=&abb^{\core, m}aa^*a^*maa*=abb^{\core, m}bua^*a^*maa*\\
&=&a^2(a^*)^2maa^*=maa^*=maaa^*a^*\\
&=&mabu(a^*)^2=mabb^{\core, m}bu(a^*)^2\\
&=&mabb^{\core, m}a^*.
\end{eqnarray*}
Hence by Lemma \ref{prop11}, $(ab)^{\core, m}= b^{\core}a^{*}$.
\qedhere
\end{enumerate}
\end{proof}

\section{Conclusion}
We have discussed a few additive properties of the weighted core and dual core inverse in a ring. In addition to that some reverse order law and mixed-type reverse order law for these inverses are investigated. This paper has not been addressed in the following problems and the possibilities for further research. 
\begin{enumerate}
	\item[$\bullet$] It will be interesting to investigate some kind of  additive properties for weighted core-EP inverse.
	
	\item[$\bullet$] To find an explicit form of the additive properties in terms of other generalized inverses.
	
	\item[$\bullet$] The reverse order law can be extended to the multiple element products in a ring.
\end{enumerate}

\medskip

\noindent{\bf{Acknowledgments}}\\
The third author is grateful to the Mohapatra Family Foundation and the College of Graduate Studies, University of Central Florida, Orlando, for their financial support for this research.

\bibliographystyle{abbrv}

\bibliography{reference}
\end{document}